\documentclass{amsart}
\usepackage{amsfonts,amssymb,amscd,amsmath,enumerate,verbatim,calc}
\usepackage{xy}

\newcommand{\CM}{Cohen-Macaulay}

\newcommand{\wrt}{with respect to}

\newcommand{\n}{\mathfrak{n} }
\newcommand{\m}{\mathfrak{m} }
\newcommand{\M}{\mathfrak{M} }

\newcommand{\q}{\mathfrak{q} }

\newcommand{\Z}{\mathbb{Z} }

\newcommand{\rt}{\rightarrow}

\newcommand{\ov}{\overline}

\newcommand{\bx}{\mathbf{x}}

\newcommand{\vdim}{\operatorname{vdim}}
\newcommand{\depth}{\operatorname{depth}}
\newcommand{\type}{\operatorname{type}}
\newcommand{\Ass}{\operatorname{Ass}}

\newcommand{\Ext}{\operatorname{Ext}}

\theoremstyle{plain}

\newtheorem{theorem}{Theorem}[section]

\newtheorem{lemma}[theorem]{Lemma}
\newtheorem{proposition}[theorem]{Proposition}

\theoremstyle{definition}

\newtheorem{s}[theorem]{}
\newtheorem{remark}[theorem]{Remark}
\newtheorem{example}[theorem]{Example}

\theoremstyle{remark}

\begin{document}

\title[$e_1 = e + 2$]{Cohen-Macaulay local rings with $e_1 = e + 2$}
\author{Tony~J.~Puthenpurakal}
\date{\today}
\address{Department of Mathematics, IIT Bombay, Powai, Mumbai 400 076}

\email{tputhen@math.iitb.ac.in}
 \begin{abstract}
In this paper we determine the possible Hilbert functions of a \CM \ local ring  of dimension $d$, multiplicity $e$ and first Hilbert coefficient $e_1$ in the case $e_1 = e + 2$.
\end{abstract}
 \maketitle
\section{introduction}
Let $(A,\m)$ be a Noetherian local ring of dimension $d$. If $M$ is an $A$-module then let $\ell(M)$ denote its length and $\mu(M)$ the number of its minimal generators. The function $H^{(1)}(A, n) = \ell(A/\m^{n+1})$ is called the \emph{Hilbert-Samuel function} of $A$ (with respect to $\m$). It is well known that there exists a polynomial $P_A(z) \in \mathbb{Q}[z]$ of degree $d$ such that $H^{(1)}(A, n) = P_A(n)$ for all $n \gg 0$. 
 The polynomial $P_A(z)$ is called the \emph{Hilbert-Samuel polynomial} of $A$.  We write
$$ P_A(z) = \sum_{i = 0}^{d}(-1)^ie_i(A)\binom{z+ d -i}{d-i}.$$
The integers $e_i(A)$ are called the \emph{$i^{th}$-Hilbert coefficient} of $A$. The zeroth Hilbert coefficient $e_0(A)$ is called the \emph{multiplicity} of $A$. 
We set $e_i = e_i(A)$ for all $i$ and $e = e_0$.

The graded ring $G(A) = \bigoplus_{n \geq 0}\m^n/\m^{n+1}$ is called the associated graded ring of $A$(with respect to $\m$). It's Hilbert series
$$H_A(z) = \bigoplus_{n \geq 0} \ell(\m^n/\m^{n+1})z^n  = \frac{h_A(z)}{(1-z)^d} \quad  \text{where} \ h_A(z) \in \mathbb{Z}[z].$$
If $f$ is a polynomial we use $f^{(i)}$ to denote the $i$-th formal derivative of $f$. It is
easy to see that $e_{i} = h_A(z)^{(i)}(1)/i! $ for $0 \leq i \leq d$. It is also 
convenient to set
$e_{i} = h_A(z)^{(i)}(1)/i! $ for all $i\geq 0$.  Let $h = \mu(\m) - d$ be the embedding co-dimension of $A$. It is easily seen  that $h_A(z) = 1 + hz + \text{higher terms}$. We call $h_A(z)$ the \emph{h-polynomial} of $A$. The function $H(A,n) = \ell(\m^n/\m^{n+1})$ is called the \emph{Hilbert function}  
of $A$ (with respect to $\m$).

Now assume that $A$ is \CM. Then Abhyankar proved that $e \geq h + 1$, see \cite{A}. Northcott proved that $e_1 \geq e_0 - 1$, see \cite{N}.
Itoh proved that $e_2 \geq e_1 - e + 1$, see \cite[Theorem 12]{It}.
It was observed by Sally that when the Hilbert-coefficients satisfy border values then it forces $G(A)$ to have high depth and also forces the $h_A(z)$
to have a prescribed form; see \cite{S1}, \cite{S2}, \cite{S3} and \cite{S4}. See the nice survey article \cite{V} for an introduction to this area of research. In particular see \cite{V} for classification of Hilbert functions when $e = h+1, h+2$ and $e_1 = e-1, e, e + 1$. Let $\type(A) = \ell(\Ext^d_A(k, A))$ be the \CM-type of $A$.
For classification of Hilbert functions when $e = h + 3$ and $\type(A) < h$, see \cite{RV}.

In this paper we describe all Hilbert functions that can possibly occur if $e_1 = e + 2$. It turns out that in the process we also have to describe
Hilbert functions that can occur when $e_2 = e_1 - e + 1 = 3$. It was conjectured by Valla that if $e_2 = e_1 - e + 1$ then $G(A)$ is \CM \ and 
$\deg h_A(z) \leq 2$.  This conjecture is true when $\dim A \leq 1$. 
 A counter-example  to this conjecture was found by Wang in the case $e_2 = e_1 - e + 1 = 3$  and $\dim A = 2$ ; see \cite[3.10]{CPR}. 
We prove
\begin{theorem}\label{e2}
Let $(A,\m)$ be a \CM \ local ring  of dimension $d$. Set $G = G(A)$. If $e_2 = e_1 - e + 1 = 3$ then one of the following cases occur
\begin{enumerate}[\rm (i)]
\item
$e = h + 4$, $G$ is \CM \  and $h_A(z) = 1 + hz + 3z^2$.
\item
$e = h + 3$, $d \geq 2$, $\depth G = d - 2$, $\type(A) = h$ and 

$\displaystyle{h_A(z) = 1 + hz + 3z^3 - z^4.}$
\end{enumerate} 
\end{theorem}
 Note that in Theorem \ref{e2} the $h$-polynomial of $A$ completely determines \\ $\depth G(A)$. In section 7 we give examples of \CM \ local rings 
having Hilbert functions as described in \ref{e2}.

As a consequence of Theorem \ref{e2} we can completely classify Hilbert functions that can occur when $e_1 = e + 2$.
\begin{theorem}\label{e1}
Let $(A,\m)$ be a \CM \ local ring  of dimension $d$. Set $G = G(A)$. If $e_1 = e + 2$ then one of the following cases occur
\begin{enumerate}[\rm (i)]
\item 
$e = h + 2$, $\type(A) = h$, $\depth G = d - 1$ and $h_A(z) = 1 + hz + z^4$.
\item
$e = h + 4$, $G$ is \CM \  and $h_A(z) = 1 + hz + 3z^2$.
\item
$e = h + 3$, $e_2 = 4$, $\depth G \geq d - 1$ and $h_A(z) = 1 + hz + z^2 + z^3$.
\item
$e = h + 3$, $e_2 = 3$, $d \geq 2$, $\depth G = d - 2$, $\type(A) = h$ and 

$\displaystyle{h_A(z) =  1 + hz + 3z^3 - z^4.}$
\end{enumerate} 
\end{theorem}
In section 7 we give examples of \CM \ local rings 
having Hilbert functions as described in \ref{e1}.

We now describe in brief the contents of this paper. In section two we describe some preliminary results that we need. In section three we discuss the case of Theorem \ref{e2} when $d = 2$. The most difficult case of Theorem \ref{e2} is the case when $d = 3$. This is done in section four. 
In section five we prove Theorem \ref{e2}. We prove Theorem \ref{e1} in section six. Finally in section seven we give examples which illustrate our results. 
\section{preliminaries}
 In this section we discuss some preliminaries that we need. 
In this paper all rings are Noetherian and all modules are assumed to be finitely generated.
Let $(A,\m)$ be
 a local ring of dimension $d$ with residue field $k = A/\m$. 

\s If $a$ is a non-zero
 element of $A$ and if $j$ is the largest integer such that $a \in \m^{j}$,
then we let $a^*$ denote the image of $a$
 in $\m^{j}/\m^{j+1}$.

\s An element $x \in \m$ is said to be  \emph{$A$-superficial} if there exists $c$ and $n_0$ such that
$(\m^{n+1} \colon x)\cap \m^c  = \m^n$ for all $n \geq n_0$. Superficial elements exist when the residue field of $A$ is infinite. 
If $\depth A > 0$ and $x$ is $A$-superficial then $x$ is $A$-regular and furthermore $(\m^{n+1} \colon x) = \m^n$ for all $n \gg 0$.

 \begin{remark}
\label{Pr0}
 If the residue field of $A$ is finite then we resort to the
standard trick to replace $A$ by $A' = A[X]_S$ where $S =  A[X]\setminus \m A[X]$. The residue
field of $A'$ is $k(X)$, the field of rational functions over $k$.
Furthermore
\begin{equation*}
H(A',n) = H(A,n) \ \  \forall n \geq 0 \ \  \text{and} \  \   \depth G(A') = \depth G(A). 
 \end{equation*}
\end{remark}

\s \label{exist-sup} Assume the residue field of $A$ is infinite. We need to recall the construction of superficial element as it is used in the proof of  Theorem \ref{depth0}.  Set $G = G(A)$ and let
$\M$ be its maximal homogeneous ideal. Set $V$ to be the $k$-vector space $\m/\m^2$. If $P$ is a prime ideal of $G$ and $P \neq \M$ then note $P \cap V \neq V$. Then as $k$ is infinite we get
\[
V \neq \bigcup_{P \in \Ass{G}, P\neq \M} P \cap V.
\]
If $x \in \m$ is  such that 
\[
x^* \in V \setminus \bigcup_{P \in \Ass{G}, P\neq \M} P \cap V.
\]
then $x$ is $A$-superficial.

\s \label{mod-sup}Let $x \in \m$ be $A$-superficial and $A$-regular. Set $B = A/(x)$. Also set $b_n = \ell((\m^{n+1} \colon x)/\m^n)$ and
$b_x(z) = \sum_{n \geq 0}b_n z^n$. Then $b_x(z) \in \Z[z]$. Furthermore we have
\[
h_A(z) = h_B(z) - (1-z)^db_x(z).
\]
So $e_i(A) = e_i(B)$ for $0 \leq i \leq d - 1$. Also $x^*$ is $G(A)$-regular if and only if $b_x(z) = 0$. In this case we have $G(B) = G(A)/(x^*)$.

\s A
sequence $x_1,x_2,\ldots,x_r$ in a local ring $(A,\m)$ is said to
be an $A$-\emph{superficial sequence}  if $x_1$ is $A$-superficial
 and $x_i$ is  $A/(x_1,\ldots,x_{i-1})$-superficial for
$2\leq i \leq r$.

\s\label{SD} Assume $A$ is \CM \ and $x_1,\ldots, x_r$ is an $A$-superficial sequence with $r \leq d$. Set $B = A/(x_1,\ldots, x_r)$. Then
\begin{enumerate}
\item
$\depth G(A) \geq r $ if and only if $x_1^*,\ldots, x_r^*$ is $G(A)$-regular. In this case we have $G(B) = G(A)/(x_1^*,\ldots, x_r^*)$.
\item
(\emph{Sally descent:}) $\depth G(A) \geq r + 1$ if and only if $\depth G(B) \geq 1$.
\end{enumerate}

\s \label{SA} Let $(A,\m)$ be \CM \ and suppose $x_1,\ldots, x_d$ be a maximal $A$-superficial sequence. Set $J = (x_1,\ldots, x_d)$.
Then $e = h+ 1+ \ell(\m^2/J\m)$.

\s \label{dim1}
 Let  $\dim A = 1$  and $A$ is \CM. Let $x$ be $A$-superficial.
 Set $\rho_n = \lambda(\m^{n+1}/x\m^n)$.   We have
   \begin{enumerate} [\rm (1)]
    \item
    $ H(A,n) = e - \rho_n$.
    \item
    If $\deg h_A(z) = s$ then
   $h_A(z) = 1 + \sum_{i = 0}^{s}(\rho_{i-1} - \rho_{i})z^i$
    \item
    $ e_i = \sum_{j\geq i -1}\binom{j}{i-1}\rho_j \geq 0 $ for all $i \geq 1$.
   \end{enumerate}

\s Assume $\depth A > 0$. For $i \geq 1$ set
\[
\widetilde{\m^i} = \bigcup_{l \geq 1}(\m^{i+l} \colon \m^l).
\]
Then $\{ \widetilde{\m^n} \}_{n \geq 0}$ is called the Ratliff-Rush filtration of $A$ (with respect to $\m$). 
We have
\begin{enumerate} [\rm (1)]
    \item
$\widetilde{\m^i}\widetilde{\m^j} \subseteq \widetilde{\m^{i+j}}$ for all $ i, j \geq 0$.
\item
$\widetilde{\m^i} = \m^i$ for all $i \gg 0$.
\item
If $x$ is $A$-superficial then $(\widetilde{\m^{i+1}} \colon x) = \widetilde{\m^i}$ for all $i \geq 0$.
\item
$\depth G(A) > 0$ if and only if $\widetilde{\m^i} = \m^i$ for all $i \geq 1$.
\end{enumerate}

\s \label{ratliff-hilb} Let $\depth A > 0$. Let $\widetilde{G}(A) = \bigoplus_{n \geq 0}\widetilde{\m^n}/\widetilde{\m^{n+1}}$ be the associated graded ring 
of the Ratliff-Rush filtration. Let its Hilbert series be
\[
\sum_{n \geq 0}\ell\left( \widetilde{\m^n}/\widetilde{\m^{n+1}}\right)z^n = \frac{\widetilde{h}_A(z)}{(1-z)^d} \quad  \text{where} \ \widetilde{h}_A(z) \in \mathbb{Z}[z].
\]
Set $r(z) = \sum_{n \geq 0} \ell(\widetilde{\m^{n+1}}/\m^{n+1})z^n$. Then $r(z) \in \Z[z]$ and
\[
h_A(z) = \widetilde{h}_A(z) + (1-z)^{d+1}r(z).
\]

\s\label{bella} (\cite[Theorem 3]{It}) Let $(A,\m)$ be a \CM \ local ring of dimension two.  Suppose  $x,y$ is an $A$-superficial sequence. Set $J = (x,y)$. 
Let $\sigma_j = \ell(\widetilde{\m^{j+1}}/J\widetilde{\m^j})$ for $j \geq 0$. Then
\[
e_1 = \sum_{j \geq 0} \sigma_j  \quad \text{and} \  e_2 = \sum_{j \geq j} j\sigma_j. 
\]
Furthermore $\widetilde{h}_A(z) = 1 + \sum_{i \geq 1}(\sigma_{i-1} - \sigma_i)z^i$.

We will need the following result in the proof of Theorem \ref{depth0}.
\begin{proposition}\label{r-3}
Set $G = G(A)$ and $\widetilde{G} = \widetilde{G}(A)$.
 Assume $\widetilde{G}$ is \CM. Then
\begin{enumerate}[\rm (1)]
\item
$G$ is generalized \CM.
\item
$\dim G/P = d$ for all minimal primes $P$ of $G$.
\end{enumerate}
\end{proposition}
\begin{proof}
(1) We note that $\widetilde{G}$ is finitely generated as a $G$-module. Furthermore  the natural map $G \rt \widetilde{G}$ induces an exact sequence of $G$-modules
\[
0 \rt U \rt G \rt \widetilde{G} \rt V \rt 0,
\]
where $U, V$ are of finite length. Let $\M$ be the maximal homogeneous ideal of $G$. If $P$ is a prime ideal in $G$ with $P \neq \M$ then $G_P \cong \widetilde{G}_P$. It  follows that $G_P$ is \CM \ for all $P \neq \M$. So $G$ is generalized \CM. 

(2) Let $P$ be a minimal prime of  $G$. Then $P \neq \M$. Furthermore $G_P \cong \widetilde{G}_P$. So $P$ is an associate prime of
$\widetilde{G}$. The result now follows from \cite[2.1.2(a)]{BH}.
\end{proof}

\s \label{quot} Assume $(A,\m)$ is \CM \ of dimension $d \geq 2$. Let $x_1,\ldots, x_d$ be an $A$-superficial sequence. Set $J = (x_1,\ldots, x_d)$ and $(B,\n) = (A/(x_1), \m/(x_1))$.  If $I$ is an ideal in $A$ then set $\ov{I}$ to be its image in $B$. Notice $\ov{\widetilde{\m^i}} \subseteq \widetilde{\n^i}$. So we have a natural maps
 $$\eta_i \colon \widetilde{\m^{i+1}}/ J\widetilde{\m^i}  \rt \widetilde{\n^{i+1}}/\ov{J} \widetilde{\n^i} \quad \text{for all} \ i \geq 0. $$
We show
\begin{proposition}\label{nn}
(with setup as in \ref{quot})
\begin{enumerate}[\rm (1)]
\item
If $\ov{\widetilde{\m^s}} = \widetilde{\n^s}$ for some $s$ then $\eta_s$ is injective.
\item
If $\ov{\widetilde{\m^j}} = \widetilde{\n^j}$ for $j = s, s+1$ then $\eta_s$ is bijective.
\end{enumerate}
\end{proposition}
\begin{proof}
(1) Suppose $\eta_s( p + J \widetilde{\m^s}) = 0$. Then $\ov{p} = \sum_{ t = 2}^{d}a_tx_t$ for some $a_t \in \widetilde{\n^s}$. As $\ov{\widetilde{\m^s}} = \widetilde{\n^s}$ there exists $b_t \in \widetilde{\m^s}$ with $\ov{b_t} = a_t$ for $t = 2,\ldots, d$. So
\[
p = \sum_{t =2}^{d} b_tx_t +   rx_1 \quad \text{for some} \ r \in A.
\]
We get  $xr \in \widetilde{\m^{s+1}}$. So $r \in \widetilde{\m^s}$. Thus $p \in J \widetilde{\m^s}$. It follows that $\eta_s$ is injective.

(2) As $\ov{\widetilde{\m^{s+1}}} = \widetilde{\n^{s+1}}$ then clearly $\eta_s$ is surjective. As $\ov{\widetilde{\m^s}} = \widetilde{\n^s}$ we get by (1) that $\eta_s$ is injective. So $\eta_s$ is bijective. 
\end{proof}
The following result is needed in the  proof of Theorem \ref{depth0}.
\begin{lemma}\label{m3J}
Let $(A,\m)$ be a \CM \ local ring of dimension three with an infinite residue field. Let $x,y$ be an $A$-superficial sequence. Let $z$ be $A\oplus A/(x,y)$-superficial (note $z$ is $A$-superficial and $x,y,z$ is an $A$-superficial sequence). Set $J = (x,y,z),  (B, \n) = (A/(x), \m/(x)) $ and $\ov{J} 
= (y,z)B$. Consider the natural complex
$$ 0 \rt \frac{ (\m^3  \colon x)}{\m^2} \xrightarrow{\alpha} \frac{ \m^3}{J\m^2} \xrightarrow{\pi} \frac{\n^3}{\ov{J} \n^2} \rt 0$$
where $\pi$ is the natural surjection and $\alpha(a + \m^2) = xa + J\m^2$.  Then the above complex is exact.
\end{lemma}
\begin{proof}
Clearly $\pi$ is surjective. Let $\pi(a + J\m^2) = 0$. It follows that $a = y\xi_2 + z \xi_2^\prime + xr$ where $\xi, \xi^\prime \in \m^2$ and $r \in A$. Note $r \in (\m^3 \colon x)$. Furthermore $\alpha(r + \m^2) = a + J\m^2$. 

It remains to prove $\alpha$ is injective. Suppose $\alpha(t + \m^2) = 0$. It follows that
$tx  = ax + by + cz$ where $a,b,c \in \m^2$. Set $(C, \q) = (A/(z), \m/(z))$. We note that $J$ is a minimal reduction of $\m$. So $(x,y)C$ is a 
minimal reduction of $\q$.   In $C$ we get $\ov{t}x = \ov{a}x + \ov{b} y$. As $x,y$ is a regular sequence in $C$ we get
$\ov{t}- \ov{a} = \delta y$ and $\ov{b} = \delta x$ for some $\delta \in A$. As $(x,y)C$ is a minimal reduction of $\q$ and $\ov{b} \in \q^2$, by analyticity (see \cite[Lemma 2]{NR}) we get $\delta \in  \m$.
Thus $t = a + y\delta + rz$ for some $r \in A$. Note $t \in \widetilde{\m^2}$. So $rz \in \widetilde{\m^2}$. As $z$ is $A$-superficial we get
$r \in \m$. It follows that $t \in \m^2$. Thus $\alpha$ is injective. 
\end{proof}
\section{$e_2 = e_1 - e + 1 = 3$ and $\dim A = 2$}
To prove this result we need the following:
\begin{proposition}\label{dim2}
Let $(A,\m)$ be a \CM  \ local ring of dimension two and with an infinite residue field. Let $x,y$ be an $A$-superficial sequence. Set $(B,\n) = (A/(x), \m/(x))$ and $J = (x,y)$.  Assume $e_2 = e_1 - e_0 + 1 = 3$.  Set $G = G(A)$. Then
\begin{enumerate}[\rm(1)]
\item
$\depth G = 0$ or $2$. If $\depth G = 2$ then $h_A(z) = 1 + hz + 3z^2$. In particular $e = h + 4$ if $\depth G = 2$.
\item
If $\depth G = 0$ then
\begin{enumerate}[\rm (a)]
\item
$\widetilde{\m^{n+1}}  = J \widetilde{\m^n}$ for all $n \geq 2$.
\item
$\widetilde{G}(A)$ is \CM.

\item
 $\widetilde{\m^2} \neq \m^2$. 

\item 
 $\ell(\widetilde{\m^2}/\m^2) = 1$, $\ell(\m^2/J\m) = 2$  and $e = h + 3$.
\item
 $\widetilde{h}_A(z) = 1 + (h-1)z + 3z^2$.
\item 
$h_B(z) = 1 + hz + z^2 + z^3$.

\item
 $\widetilde{\m^j} = \m^j$ for all $j \geq 3$. Furthermore $\ell(\m^3/J\m^2) = 2$ and $\m^4 = J\m^3$.

\item
$h_A(z) = 1 + hz + 3z^3 - z^4$.

\item
$\type(A) = h$.
\end{enumerate}
\end{enumerate}
\end{proposition}
\begin{proof}
Let $x,y$ be an $A$-superficial sequence  \wrt \ $\m$. 
(1) If $\depth G \geq 1$ then we have $e_2(B) = e_1(B) - e_0(B) + 1$. Using \ref{dim1} it follows that $\n^3 = y\n^2$ and $\ell(\n^2/y\n)  = 3$.  So by \cite[2.1]{S2} we have that $\depth G_\n(B) = 1$ and $h_B(z) = 1 + hz + 3z^2$. So by Sally descent we have $\depth G = 2$. 
Also $h_A(z) = h_B(z)$.

(2) For $j \geq 0$ set $\sigma_j = \ell(\widetilde{\m^{j+1}}/J\widetilde{\m^{j}})$. Then by \ref{bella} we have $e_1 = \sum_{j \geq 0} \sigma_j$ and $e_2 = \sum_{j \geq 1}j\sigma_j$. Since $e_2 = e_1 - e_0 + 1$ we get $\sigma_j = 0$ for $j \geq 2$ and
$e_2 = \sigma_1 = \ell(\widetilde{\m^{2}}/J\m) = 3$. 

(a) This follows as $\sigma_j = 0$ for $j \geq 2$.\\
(b) We have $\widetilde{\m^2} \cap J = J \m $ and for $i \geq 2$ we also have $\widetilde{\m^{i+1}} \cap J = J\widetilde{\m^i}$ (since $\sigma_j = 0$ for $j \geq 2$). Thus by Valabrega-Valla theorem for filtration's we get that $\widetilde{G}(A)$ is \CM \ (see \cite[3.5]{HM}).

(c) If $\widetilde{\m^2} = \m^2$ then as $\sigma_2 = 0$ we get $\widetilde{\m^3} = J\m^2 $. It follows that $\widetilde{\m^3} = \m^{3}$. As $\sigma_j = 0$ for $j \geq 2$ inductively one can show that $\widetilde{\m^j} = \m^j$ for $j \geq 2$. It follows that $\depth G \geq 1$, a contradiction. Therefore $\widetilde{\m^2} \neq \m^2$.

(d) Note 
\[
3 = e_2 = \ell(\widetilde{\m^{2}}/J\m) = \ell(\widetilde{\m^2}/\m^2) + \ell(\m^2/J\m).
\]
It follows that $\ell(\m^2/J\m) \leq 2$. If $\ell(\m^2/J\m) \leq 1$ then by  \cite{rv1} or \cite{W} we have $\depth G \geq 1$, a contradiction. Thus $\ell(\m^2/J\m) = 2$. It follows that $\ell(\widetilde{\m^2}/\m^2) = 1$. We also have $e = h + 3$ by \ref{SA}.

(e) This follows from \ref{bella}. 
 
(f) As $e_i(B) = e_i$ for $i \leq 1$, \ref{mod-sup} we get $e_1(B) = e_0(B) + 2$. 
 Note $e_1(B) = \sum_{j\geq 0}\ell(\n^{j+1}/y\n^j)$. Also $\ell(\n^2/y\n) = \ell(\m^2/J\m) = 2$. It follows that $\ell(\n^3/y\n^2) = 1$ and $\n^4 = y \n^3$. 
 Thus $e_2(B) =  \sum_{j\geq 1}j\ell(\n^{j+1}/y\n^j) = 4$. The formula for $h_B(z)$ follows from \ref{dim1}.

(g) By  \ref{mod-sup} we have
\[
e_2 = e_2(B) - \sum_{n \geq 2} \ell((\m^{n+1} \colon x)/\m^n).
\]
So $ \sum_{n\geq 2} \ell((\m^{n+1} \colon x)/\m^n) = 1$.  We have an exact sequence, see \cite[p.\ 305]{rv2}
\begin{equation*}
0 \rt \frac{(\m^3 \colon x)}{\m^2} \rt \frac{\m^3}{J\m^2} \rt \frac{\n^3}{y\n^2} \rt 0. \tag{*}
\end{equation*}
If $(\m^3 \colon x) = \m^2$ then $\ell(\m^3/J\m^2) = 1$. So by a result due to Huckaba \cite{H} we get $\depth G \geq 1$, a contradiction. So $(\m^3 \colon x) \neq \m^2$.
It follows that $(\m^{n+1} \colon x) = \m^n $ for $n \geq 3$.  
For all $i \geq 0$ we have an exact sequence, see \cite[2.6]{Pu5},
 \begin{equation*}
  0 \rt \frac{(\m^{i+1}\colon x)}{\m^{i}} \rt \frac{\widetilde{\m^i}}{\m^i}  \rt \frac{\widetilde{\m^{i+1}}}{\m^{i+1}} 
 \end{equation*}
 It follows that $\widetilde{\m^j} = \m^j$ for $j \geq 3$. By (*) we get $\ell(\m^3/J\m^2) = 2$. Also by (a),  $\widetilde{\m^4} = J\widetilde{\m^3}$. As $\widetilde{\m^j} = \m^j$ for $j \geq 3$ we get $\m^4 = J\m^3$.

(h) This follows from \ref{ratliff-hilb}.

(i) If $\type(A) < h$ then it follows from \cite[2.3]{RV} that $\ell(\m^3/J\m^2) \leq 1$ and so by a result of Huckaba \cite{H} we get $\depth G \geq 1$ which is a contradiction.
We show that $\type(A) > h$ is also not possible. Set $(C,\q) = (A/(x,y) , \m/(x,y))$. As $\m^3 \subseteq J$ we  get $\q^3 = 0$. 
Also as $\m^2 \nsubseteq J$ we get $\q^2 \neq 0$. So
$h_C(z) = 1 + hz + 2z^2$.  Note $\type(A) = \type(C)$. As $\q^2 \neq 0 $ we get $\type(C) \leq h + 1$. If $\type(C) = h + 1$ then
$\q = (a, K)$ where $K \subseteq (0 \colon \q)$. Then $\q^2 = (a^2)$ a contradiction as $\mu(\q^2) = 2$.  Thus $\type(C) = h$.
\end{proof}

\section{  $e_2 = e_1 - e + 1 = 3$ and $\dim A = 3$ }
The most difficult case of Theorem \ref{e2} occurs when $\dim A = 3$. We prove it separately.
\begin{theorem}\label{depth0}
Let $(A,\m)$ be a \CM \ local ring of dimension $3$ and infinite residue field. Assume $e_2 = e_1 - e + 1 = 3$.  Then one of the following cases occur
\begin{enumerate}[\rm (i)]
\item
$e = h + 4$, $G$ is \CM \  and $h_A(z) = 1 + hz + 3z^2$.
\item
$e = h + 3$, $d \geq 2$, $\depth G = 1$, $\type(A) = h$ and 

$\displaystyle{h_A(z) = 1 + hz + 3z^3 - z^4.}$
\end{enumerate} 
\end{theorem}
\begin{proof}
 Let $x$ be any $A$-superficial element. Set $(B,\n) = (A/(x), \m/(x))$.  Then $e_i(A) = e_i(B)$ for $i = 0,1,2$. By \ref{dim2} we get that
$e = h + 4$ or $e = h + 3$. If $e = h + 4$ then $G(B)$ is \CM. So by Sally descent we get $G(A)$ is \CM. Furthermore
$h_A(z) = h_B(z) = 1 + hz + 3z^2$. If $e = h + 3$ then $G(B)$ has depth $0$. Note $\type(A) = \type(B) = h$. By \ref{SD} we get $\depth G(A) \leq 1$. If $\depth G(A) = 1$ then $x^*$ is $G(A)$-regular and $G(A)/(x^*) = G(B)$. So $h_A(z) = h_B(z) = 1 + hz + 3z^3 - z^4$ (by \ref{dim2}).

Suppose if possible $\depth G(A) \neq 1$. Then by the above argument we get $\depth G(A) = 0$. We show:

\emph{Claim}(1) : $\ell(\widetilde{\m^2}/\m^2) = 1$.

We have an exact sequence for all $i \geq 1$ (see \cite[2.9]{Pu5})
\[
 0 \rt \frac{(\m^{i+1} \colon x)}{\m^i} \rt \frac{\widetilde{\m^{i}}}{\m^{i}}  \rt \frac{\widetilde{\m^{i+1}}}{\m^{i+1}}\rt
 \frac{\widetilde{\n^{i+1}}}{\n^{ i + 1}} \tag{*}
\]
As $\widetilde{\m} = \m$ we get an inclusion $\widetilde{\m^2}/\m^2 \rt \widetilde{\n^2}/\n^2$. Furthermore as $\widetilde{\n^j} = \n^j$ for 
$j \geq 3$  (see \ref{dim2}.2(g)) we get surjections $\widetilde{\m^j}/\m^j \rt \widetilde{\m^{j+1}}/\m^{j+1}$ for $j \geq 2$. Thus $\ell(\widetilde{\m^2}/\m^2) \leq 1$. Furthermore if $\widetilde{\m^2} = \m^2$ we get $\widetilde{\m^j} = \m^j $ for all $j$ and so $\depth G(A) > 0$ which is a contradiction.
Thus $\ell(\widetilde{\m^2}/\m^2) =1$.

\emph{Claim}(2): $\widetilde{G}(A)$ is \CM.

In the inclusion $\widetilde{\m^2}/\m^2 \rt \widetilde{\n^2}/\n^2$ both modules have same length. So $\ov{\widetilde{\m^2}}  = \widetilde{\n^2}$. Furthermore as $\widetilde{\n^j} = \n^j$ for $j \geq 3$ we get $\ov{\widetilde{\m^j}}  = \widetilde{\n^j}$ for $j \geq 3$. 
Thus the Ratilff-Rush filtration on $A$ behaves well \wrt \ superficial element $x$. So
 $\widetilde{G}(A)/(x^*) = \widetilde{G}(B)$. As $\widetilde{G}(B)$ is \CM \  (see \ref{dim2}.2(b) ) and $x^*$ is $\widetilde{G}(A)$-regular we get $\widetilde{G}(A)$ is \CM.

\emph{Claim}(3):   $G(A)$ is generalized \CM.

 By \ref{r-3} we get $G(A)$ is generalized \CM. 

\emph{Claim}(4):  $\widetilde{\m^j} = \m^j $ for all $j \geq 3$. Furthermore $\ell((\m^3 \colon x)/\m^2) = 1$.

 As $\widetilde{\n^3} = \n^3$, by (*) we have an exact sequence
\begin{equation*}
 0 \rt \frac{(\m^{3} \colon x)}{\m^2} \rt \frac{\widetilde{\m^2}}{\m^{2}}  \rt \frac{\widetilde{\m^{3}}}{\m^{3}}\rt 0. \tag{**}
\end{equation*}
Suppose if possible $\widetilde{\m^3} \neq \m^3$.  Then as $\ell(\widetilde{\m^2}/\m^2) = 1$ we get
$(\m^3 \colon x) = \m^2$ and $\ell(\widetilde{\m^3}/\m^3) = 1$. We note that we chose $x$ to be \emph{any} $A$-superficial element.
As $G(A)$ has depth zero and also as it is generalized \CM  \ it follows that
\[
\Ass(G(A)) = \{ \M, P_1, \ldots, P_s \};
\]
where $P_1, \ldots, P_s$ are minimal primes of $G(A)$ and $\M$ is the maximal homogeneous ideal of $G(A)$. By \ref{r-3} we get $\dim G(A) /P_i =   3$ for all $i$. Let $V$ be the $k$-vector space
$\m/\m^2$.  If $W$ is a $k$ vector  space then we set $\vdim W$
to be its dimension as a $k$-vector space.  If $V \subseteq P_i$ then $P_i = \M$, a contradiction. In particular $\vdim V \cap P_i \leq \vdim V -1$ for all $i$.

\textit{sub-claim:} $\vdim P_i \cap V \leq \vdim V - 2$ for all $i$.\\
Suppose for some $i$ we have $\vdim P_i \cap V =  \vdim V - 1$. Then there exists $t \in V$ such that $P_i \cap V + k t = V$. So we get
$P_i + G(A)t = \M$. It follows that $(t)$ generates the maximal homogeneous ideal in $G(A)/P_i$. So $\dim G(A)/P_i \leq 1$ a contradiction. 
 Thus we have proved the sub-claim. 

As $k$ is infinite there exists $k$-linearly independent vectors $u^*, v^*$ in $V$ such that if $H  = ku^* + kv^*$ then $H \cap P_i = 0$ for all $i$. It follows that $u - \lambda v$ is an $A$-superficial sequence for any unit $\lambda$ (see \ref{exist-sup}). Furthermore $u, v$ are $A$-superficial.

Suppose $\widetilde{\m^2} = \m^2 + (a)$. Let $x$ be any $A$-superficial element. If $xa \in \m^3$ then $a \in (\m^3 \colon x) = \m^2$ a contradiction. Thus $\ov{xa}$ is a non-zero element in $\widetilde{\m^3} /\m^3$. In particular $\ov{ua}, \ov{va}$ are non-zero in $\widetilde{\m^3} /\m^3$. As $\widetilde{\m^3} /\m^3 \cong k$ we get that there is a unit $\alpha$  such that  $\ov{ua} = \alpha \ov{va}$.
So $(u - \alpha v)a \in \m^3$. By construction $u - \alpha v$ is an $A$-superficial element. So $a \in (\m^3 \colon  u - \alpha v) = \m^2$ a contradiction. Thus $\widetilde{\m^3} = \m^3$. As we have surjections  $\widetilde{\m^j}/\m^j \rt \widetilde{\m^{j+1}}/\m^{j+1}$ for $j \geq 3$ we also get $ \widetilde{\m^j} = \m^j$ for $j \geq 4$.  By exact sequence (**) we also get $\ell((\m^3 \colon x)/\m^2 ) = 1$. This proves claim (4)

Let $x,y,z$ be an $A$-superficial sequence. Set $J = (x,y,z)$.

\emph{Claim} (5): $\widetilde{\m^{i+1}} = J \widetilde{\m^i}$ for $i \geq 2$.  In particular $\m^3 = J \widetilde{\m^2}$ and $\m^4 = J \m^3$.

We have the natural map $\eta_i \colon \widetilde{\m^{i+1}}/ J \widetilde{\m^i} \rt  \widetilde{\n^{i+1}}/ (y,z)\widetilde{\n^i}$  for all $i \geq 1$. By Claim(1) we have $\ov{\widetilde{\m^i}} = \widetilde{\n^i}$ for all $i \geq 1$. From \ref{nn} we get that $\eta_i$ is an isomorphism for $i \geq 1$. The result now follows from \ref{dim2}.(2)(a)  and Claim (4).

 \emph{Claim}(6) : $\ell(\m^2/J\m) = 2$ and  $\ell(\m^3/J\m^2) = 3$.
 
As $e = h + 3$ we get by \ref{SA} that $\ell(\m^2/J\m) = 2$.  By \ref{m3J}, for a suitable choice of $x,y,z$,  we get  a short exact sequence
\[
 0 \rt \frac{ (\m^3  \colon x)}{\m^2} \xrightarrow{\alpha} \frac{ \m^3}{J\m^2} \xrightarrow{\pi} \frac{\n^3}{\ov{J} \n^2} \rt 0.
\]
By \ref{dim2}(2)g and claim (4) we get for a choice of $x,y,z$ we have $\ell(\m^3/J\m^2) = 3$. By \cite{Pu-r};  $\ell(\m^3/J\m^2)$ is an invariant of $A$ and does not depend on choice of minimal reduction $J$. Thus  for any $A$-superficial sequence $x,y,z$ we have
$\ell(\m^3/J\m^2) = 3$.

\emph{Claim}(7):  $S = A/\m \oplus \m/(J + \m^2) \oplus \m^2/J\m \oplus \m^3/J\m^2$ is a finite dimensional homogeneous $k$-algebra.

In fact $S = G(A)/(x^*, y^*, z^*)$ (use $\m^3 \subseteq J\m$ and $\m^4 = J\m^3)$.

\emph{Claim}(8):  $S$ violates Macaulay's bound  ( \cite[4.2.10(c) ]{BH})for Hilbert function of $k = A/\m$-algebras.

Let $f$ be the Hilbert function of $S$. We have by claim (6)
\[
f(2) = 2  = \binom{2}{2} + \binom{1}{1}.
\]
By Macaulay's bound we have 
\[
3 = f(3) \leq  f(2)^{<2>} = \binom{3}{3} + \binom{2}{2} = 2,  \quad \text{a contradiction.}
\]
Thus our assumption $\depth G(A) \neq 1$ is incorrect. So $\depth G(A) = 1$.
\end{proof}

\section{Proof of Theorem \ref{e2}}
In this section we give
\begin{proof}[Proof of Theorem \ref{e2}]
We may assume that the residue field of $A$ is infinite.
If $\dim A \leq 1$ then by \cite[6.21]{V} we get $G(A)$ is \CM  \ and $\deg h_A(z) \leq 2$. It follows that $h_A(z) = 1 + hz + 3z^2$.
When $d = \dim A \geq 2$ we prove the result by induction on $d$.
The case when $\dim A = 2$ is analyzed in \ref{dim2}. 
The case when $\dim A = 3$ is analyzed in \ref{depth0}. 

 Now assume $d = \dim A \geq 4$ and the result holds for \CM \ local rings of dimension $d - 1 \geq 3$.

Let $x$ be an $A$-superficial. Set $(B, \n) = (A/(x), \m/(x))$. 
We  consider the following two cases:\\
Case (1)  $G(A)$ is \CM.  So $G(B)$ is \CM. In this case $h_A(z) = h_B(z) = 1 + hz + 3z^2$.

Case (2) $G(A)$ is not \CM.  So $G(B)$ is not \CM \ by Sally descent. In this case we have $e = h + 3$.  By induction hypothesis we get $\depth G(B) = d - 3 \geq 1$. So by Sally descent 
we get $\depth G(A)  = d -2$. Thus  $x^*$ is $G(A)$-regular. So 
\[
h_A(z) = h_B(z) = 1 + hz + 3z^3 - z^4.
\]
\end{proof}

\section{Proof of Theorem \ref{e1}}
In this section we give
\begin{proof}[Proof of Theorem \ref{e1}]
We may assume $k = A/\m$ is infinite.
We first consider the case when $d = 0$. Clearly $G(A)$ is \CM. 
Let $h_A(z) = 1 + hz + a_2z^2 + \cdots + a_sz^s$.
We have $e_1 = e +2$. So
\[
h+ \sum_{i = 2}^{s}ia_i = 3+ h + \sum_{i = 2}^{s} a_i
\]
So we get
\[
\sum_{i = 2}^{s}(i - 1)a_i = 3
\]
So the possibilities are 
\begin{enumerate}
\item
$a_2 = 3$, $a_j = 0$ for $j \geq 3$. In this case $h_A(z) = 1 + hz + 3z^2$.
\item 
$a_2 = a_3 = 1$ and $a_j = 0$ for $j \geq 4$.  In this case $h_A(z) = 1 + hz + z^2 + z^3$.
\end{enumerate}

Next we consider the case $d = 1$.  Let $x$ be $A$-superficial. Set $\rho_n = \lambda(\m^{n+1}/x\m^n)$.    If $\deg h_A(z) = s$ then
   $h_A(z) = 1 + \sum_{i = 0}^{s}(\rho_{i-1} - \rho_{i})z^i$. Also
 $ e_1 = \sum_{j \geq 0}\rho_j  $.  Note $\rho_0 = e -1$
   In particular
\[
e_1 = e - 1 + \sum_{j \geq 1} \rho_j.
\]
Thus we have
\[
\sum_{j \geq 1} \rho_j = 3
\]
Thus $\rho_1 = 3$ or $2$ or $1$.

If $\rho_1 = 3$ then $\rho_j = 0$ for $j \geq 2$. As $\m^3 = x \m^2$ we get $G(A)$ is \CM. Furthermore
$h_A(z) = 1 + hz + 3z^2$. In particular $e = h + 4$.

If $\rho_1 = 2$ then $\rho_2 = 1$ and $\rho_j = 0$ for $j \geq 3$. So we have $h_A(z) = 1 + hz + z^2 + z^3$. In particular we have $e = h +3$.

If $\rho_1 = 1$ then $\rho_j \leq 1$ for $j \geq 2$ (see \cite[3.3]{rv2}).
So we get $\rho_2 = \rho_3 = 1$. So we have $h_A(z) = 1 + hz  + z^4$. In particular $e = h + 2$ and $G(A)$ is not \CM. It follows from \cite[3.1]{S4} that $\type(A) = h$. 

Next we consider the case $d = 2$. Let $x$ be an $A$-superficial. Set $(B, \n) = (A/(x), \m/(x))$. 
We have $e_i(A) = e_i(B)$ for $i = 0, 1$.  So we have $e = h + 2$ or $e = h + 3$ or $e = h +4$.

 If $e = h + 4$ then $G(B)$ is \CM. So by
Sally descent we get $G(A)$ is \CM.  Furthermore $h_A(z) = h_B(z) = 1 + hz + 3z^2$.

If $e = h + 2$ then by \cite{rv1} or \cite{W} we get $\depth G(A) = 1$. So $x^*$ is $G(A)$-regular. Thus $h_A(z) = h_B(z) = 1 + hz  + z^4$.
We also have
$\type(A) =  \type(B) = h$.

Now consider the case $e = h + 3$.  By Itoh's bound we have $e_2(A) \geq e_1 - e + 1 = 3$.  Set $b_i = \ell((\m^{i+1} \colon x)/\m^i)$ for $i \geq 0$. Note $e_2(B) = 4$. Then by \ref{mod-sup} we get $e_2(A) = 4 - \sum_{i \geq 0}b_i$. Thus $e_2(A) = 4$ or $3$. \\
Subcase (1) $e_2(A) = 4$. In this case $ (\m^{i+1} \colon x) = \m^i$ for all $i$. In particular $x^*$ is $G(A)$-regular. So
$\depth G(A) \geq 1$. Also $h_A(z) = h_B(z) = 1 + hz + z^2 + z^3$.\\
Subcase (2) $e_2 = 3$. Note $e_2 = e_1 - e + 1 = 3$. Also $\depth G(A) = 0 $ as $(\m^{i+1} \colon x) \neq \m^i$ for some $i$. The result now follows from \ref{dim2}.

Now assume $d \geq 3$. Let $ \bx = x_1,\ldots, x_{d-2}$ be an $A$-superficial sequence. Set $(B, \n) = (A/(\bx), \m/(\bx))$. Then
$e_i(A) = e_i(B)$ for $i = 0, 1, 2$. The result now follows from our analysis of $d = 2$ case, Sally descent and Theorem \ref{e2}.
\end{proof}

\section{examples}
In this section we give examples which shows that there exists \CM \ local rings with Hilbert functions as described in Theorems \ref{e2} and \ref{e1}.
We begin by the following 
\s \label{const} \emph{Construction:}  Let $(A,\m)$ be \CM. Set $T = A[X_1,\ldots, X_m]$ and $\n = (\m, X_1,\ldots, X_m)$. Set $R = T_\n$.
Then $G(R) \cong G(A)[X_1^*, \ldots, X_m^*]$.

\begin{example} \label{dance}
 Let $C = k[X,Y,Z]/( X^2, XY, Y^2, (X,Y,Z)^3)$. Then $C$ is Artin local and $h_C(z) = 1 + 3z + 3z^2$. Using \ref{const} we can construct for every $m \geq 1$,  \CM \ local rings $A_m$ of dimension $m$ with $h = 3$, $e= h + 4 = 7$ and $h_A(z) = h_C(z)$. Also $G(A_m)$ is  \CM.
\end{example}

\begin{example} \label{e2d0}
The following example of Wang, see \cite[3.10]{CPR},  is of a 2-dimensional \CM \ local ring with $e_2 = e_1 - e + 1 = 3$ and $\depth G(A)  = 0$.
$$A = k[[ x, y, z, u, v]]/(z^2, zu, zv, uv, yz - u^3,  xz-v^3).$$ 
We have $h_A(z) = 1 + 3z + 3z^3 - z^4$. 
Using \ref{const} we may construct \CM \ local rings $A_m$ of dimension $m + 2$
with $\depth G(A_m) = m$.
\end{example}

We note that Examples \ref{dance} and \ref{e2d0} yield \CM \ local rings with Hilbert functions as described in Theorem \ref{e2}.

\begin{example}\label{sally1}
By \cite[4.4]{S4}
\[
R = k[[ t^6, t^7, t^{11}, t^{15}, t^{16}]]
\]
has multiplicity $6$, embedding dimension $5$ and reduction number $4$. So $h_R(z) = 1 + 4z + z^4$. Using \ref{const} we may construct 
\CM \ local rings $A_m$ of dimension $m + 1$ with $\depth G(A_m) = m$ and $h_{A_m}(z) = h_R(z)$.
\end{example}

\begin{example}\label{sally2}
By \cite{S3} we get that if $e = h + 3$ and $A$ is Gorenstein then $G(A)$ is \CM \ and $h_A(z) = 1 + hz + z^2 + z^3$.
By \cite[p.\ 97]{S3};  $R = k[[X,Y,Z]]/(XZ-YZ,  XZ + Y^3 - Z^2)$  is Gorenstein of dimension one, multiplicity 5 and $h = 2$. So $h_R(z) = 1 + 2z + z^2 + z^3$.  Using \ref{const} we may construct 
\CM \ local rings $A_m$ of dimension $m + 1$ with $G(A_m)$ \CM \   and $h_{A_m}(z) = h_R(z)$.
\end{example}

\begin{example}\label{tony}
Let $(A,\m)$ be the two dimensional local ring as in \ref{e2d0}. We assume $k$ is infinite. Let $u$ be an $A$-superficial. Set $B = A/(u)$. Then by
\ref{dim2}(2)(f) we get that  $h_B(z) = 1 + 3z + z^2 + z^3$ and $\depth G(B) = 0$.  Using \ref{const} we may construct 
\CM \ local rings $A_m$ of dimension $m + 1$ with $\depth G(A_m) = m$    and $h_{A_m}(z) = h_B(z)$.
\end{example}
 
We note that by the examples constructed  yield \CM \ local rings with Hilbert functions as described in Theorem \ref{e1}. Note we have also constructed \CM \ local rings $A$ with $h_A(z) = 1 + hz + z^2 + z^3$ such that $G(A)$ is \CM \ (see \ref{sally2}) and when $\depth G(A) = d - 1$ (see \ref{tony}).

\end{document}